\documentclass[reqno,english]{article}
\usepackage{braket,hyperref}
\usepackage{amssymb,mathtools,amsthm,amsmath}
\usepackage{tikz}
\usepackage{xargs}
\usepackage{enumitem}
\usepackage[colorinlistoftodos,prependcaption,textsize=small]{todonotes}
\usepackage{verbatim}

\theoremstyle{plain}
\newtheorem{theorem}{\bf Theorem}

\newtheorem{proposition}[theorem]{\bf Proposition}

\newtheorem{lemma}[theorem]{\bf Lemma}
\theoremstyle{definition}

\usepackage[a4paper]{geometry}
\numberwithin{theorem}{section}
\numberwithin{equation}{section}

\newcommand{\Rea}{{\mathbb R}}

\newcommand{\eigen}[3]{\lambda_{\text{#1}}^{(#2)}(#3)}

\newcommand{\angles}[2]{#2^T #1}

\usepackage{authblk}

\begin{document}

\title{Garland's method for token graphs}
\author{Alan Lew\footnote{e-mail: alanlew@andrew.cmu.edu}}
 \affil{Dept. Math. Sciences, Carnegie Mellon University, Pittsburgh, PA 15213, USA}

	\date{}
	\maketitle
\begin{abstract}
    The $k$-th token graph of a graph $G=(V,E)$ is the graph $F_k(G)$ whose vertices are the $k$-subsets of $V$ and whose edges are all pairs of $k$-subsets $A,B$ such that the symmetric difference of $A$ and $B$ forms an edge in $G$. 
    Let $L(G)$ be the Laplacian matrix of $G$, and $L_k(G)$ be the Laplacian matrix of $F_k(G)$.
    It was shown by Dalf\'o et al. that for any graph $G$ on $n$ vertices and any $0\leq \ell \leq k \leq \left\lfloor n/2\right\rfloor$, the  spectrum of $L_{\ell}(G)$ is contained in that of $L_k(G)$.

    Here, we continue to study  the relation between the spectrum of $L_k(G)$ and that of $L_{k-1}(G)$. In particular, we show that, for $1\leq k\leq \left\lfloor n/2\right\rfloor$, any eigenvalue $\lambda$ of $L_k(G)$ that is not contained in the spectrum of $L_{k-1}(G)$ satisfies \[
    k(\lambda_2(L(G))-k+1)\leq \lambda \leq k\lambda_n(L(G)),
\]
    where $\lambda_2(L(G))$ is the second smallest eigenvalue of $L(G)$ (a.k.a. the algebraic connectivity of $G$), and $\lambda_n(L(G))$ is its largest eigenvalue.
Our proof relies on an adaptation of Garland's 
    method, 
    originally developed for the study of high-dimensional Laplacians of simplicial complexes.
\end{abstract}
\hspace{10pt}


\section{Introduction}

Let $G=(V,E)$ be a graph. The \emph{$k$-th token graph} of $G$, denoted by $F_k(G)$, is the graph on vertex set $\binom{V}{k}$ whose edges are the pairs $\{A,B\}$ with $|A\cap B|=k-1$ and $A\triangle B=(A\setminus B)\cup (B\setminus A) \in E$.
Token graphs were originally defined by Johns in \cite{johns1988generalized} under the name of \emph{$k$-tuple vertex graphs} (see also e.g. \cite{alavi1991double,wright1992n,alavi2002survey}). 
In \cite{audenaert2007symmetric}, they were reintroduced under the name of \emph{$k$-th symmetric powers}.  Finally, in \cite{fabila2012token}, they were introduced once again under their current name.
Token graphs also appear implicitly in the study of the ``symmetric exclusion process" on graphs, introduced by Spitzer in \cite{SPITZER1970246} (see also e.g. \cite{caputo2010proof}).
Note that for $k=0$ the graph $F_0(G)$ is just the graph with one vertex (corresponding to the empty set) and no edges, for $k=1$ we have $F_1(G)\cong G$, and, if $|V|=n$, then $F_k(G)\cong F_{n-k}(G)$ for all $0\leq k\leq n$ (see e.g. \cite{fabila2012token}).

For a symmetric matrix $M\in\Rea^{m\times m}$, we denote by $\lambda_i(M)$ its $i$-th smallest eigenvalue.
Let $L(G)$ be the Laplacian matrix of $G$, and let $L_k(G)=L(F_k(G))$ be the Laplacian of its $k$-th token graph.
The Laplacian spectrum of token graphs was studied by Dalf\'o, Duque, Fabila-Monroy, Fiol, Huemer,
Trujillo-Negrete and Mart\'inez in \cite{dalfo2021laplacian}.
In particular, in was shown in \cite{dalfo2021laplacian} that for any $0\leq \ell \leq  k\leq \lfloor n/2\rfloor$ the spectrum of $L_{\ell}(G)$ is contained in the spectrum of $L_k(G)$. 

Let $1\leq k\leq \lfloor n/2 \rfloor$, and let $\lambda$ be an eigenvalue of $L_k(G)$. We say that $\lambda$ is \emph{non-trivial} if the multiplicity of $\lambda$ as an eigenvalue of $L_k(G)$ is larger than its multiplicity as an eigenvalue of $L_{k-1}(G)$. In particular, any eigenvalue of $L_k(G)$ that is not contained in the spectrum of $L_{k-1}(G)$ is non-trivial. 
We denote the maximal eigenvalue of $L_k(G)$ by $\eigen{max}{k}{G}$, and its minimal non-trivial eigenvalue by $\eigen{min}{k}{G}$. For example, for $k=1$, we have $\eigen{max}{1}{G}=\lambda_n(L(G))$ and $\eigen{min}{1}{G}=\lambda_2(L(G))$.
Our main result consists of the following bounds on $\eigen{max}{k}{G}$ and $\eigen{min}{k}{G}$.

\begin{theorem}\label{thm:step}
Let $G=(V,E)$ be a graph with $|V|=n$, and let $2\leq k\leq \lfloor n/2\rfloor$. Then
\[
  \eigen{max}{k}{G} \leq \frac{k}{k-1} \eigen{max}{k-1}{G}
\]
and
\[
\eigen{min}{k}{G}\geq \frac{k}{k-1}\eigen{min}{k-1}{G} - k.
\]
\end{theorem}

As a consequence, we obtain the following bounds on the non-trivial spectrum of $L_k(G)$.

\begin{theorem}\label{thm:main_intro}
Let $G=(V,E)$ be a graph with $|V|=n$, and let $1\leq k\leq \lfloor n/2 \rfloor$. Let $\lambda$ be a non-trivial eigenvalue of $L_k(G)$. Then,
\[
    k(\lambda_2(L(G))-k+1)\leq \lambda \leq k\lambda_n(L(G)).
\]
\end{theorem}

Both inequalities in Theorem \ref{thm:main_intro} are tight: the lower bound is attained when $G$ is a complete balanced multi-partite graph with at least $k$ parts, and the upper bound is attained when $G$ is the union of at least $k$ disjoint cliques, all of the same size (see Section \ref{sec:examples}).

It was conjectured in \cite{dalfo2021laplacian} (see also \cite{dalfo2022algebraic,reyesalgebraic}) that for any graph $G$ on $n$ vertices and any $1\leq k\leq n-1$, $\lambda_2(L(G))=\lambda_2(L_k(G))$. In fact, as mentioned in \cite{ouyang2019computing},
this follows as a special case of  Aldous' spectral gap conjecture, proved by Caputo, Liggett and Richthammer (see  \cite[Section 4.1.1]{caputo2010proof}). Note that, in the special case when $\lambda_2(G)\geq k$, the equality $\lambda_2(L(G))=\lambda_2(L_k(G))$ follows immediately from the lower bound in Theorem \ref{thm:main_intro}.

Our proof of Theorem \ref{thm:step} relies on an adaptation of Garland's ``local to global" method (\cite{garland1973p}, see also \cite{ballmann1997l2,zuk2003property}). In its original form, Garland's method relates between the spectrum of a high-dimensional Laplacian matrix on a simplicial complex to the Laplacian spectra of certain subgraphs of the complex.
In \cite{aharoni2005eigenvalues}, Aharoni, Berger and Meshulam developed a ``global version" of Garland's argument, relating the spectrum of a high-dimensional Laplacian matrix on the clique complex of a graph $G$ to the Laplacian spectrum of $G$. This relation was later extended in \cite{lew2020spectral} to more general classes of simplicial complexes. Our argument here can be seen as an analogue of the argument in \cite{aharoni2005eigenvalues}, and is motivated by the similarity between the Laplacian of the $k$-th token graph of a graph $G$ and the $(k-1)$-dimensional Laplacian of the clique complex of $G$.

The paper is organized as follows. In Section \ref{sec:prelims} we present some background material on Laplacian matrices and on the Laplacian spectrum of token graphs. In Section \ref{sec:garland} we prove our main results, Theorems \ref{thm:step} and \ref{thm:main_intro}. In Section \ref{sec:examples} we present extremal examples showing the sharpness of Theorem \ref{thm:main_intro}.

\section{Preliminaries}\label{sec:prelims}

Let $G=(V,E)$ be a graph with $|V|=n$. For convenience, we will assume $V=[n]$.
For a vertex $v\in V$, let $d_G(v)=|\{e\in E:\, v\in e\}|$ be the \emph{degree} of $v$ in $G$.
The \emph{Laplacian matrix} $L(G)\in \Rea^{n\times n}$ is defined as
\[
    L(G)_{u,v}= \begin{cases}
            d_G(v) & \text{ if } u=v,\\
            -1 & \text{ if } \{u,v\}\in E,\\
            0 & \text{ otherwise.}
    \end{cases}
\]

For $0\leq \ell\leq k\leq n$, let $B_{n,k,\ell}\in \Rea^{\binom{n}{k}\times \binom{n}{\ell}}$ be a matrix with rows indexed by the $k$-subsets of $[n]$ and columns indexed by its $\ell$-subsets, with elements
\[
    (B_{n,k,\ell})_{\sigma,\eta}= \begin{cases}
    1 & \text{ if } \eta\subset \sigma,\\
    0 & \text{ otherwise,}
    \end{cases}
\]
for $\sigma\in\binom{[n]}{k}$ and $\eta\in \binom{[n]}{\ell}$.
It is well known (see e.g. \cite{gottlieb1966certain,graver1973module,wilson1973necessary}) that for $k\leq \lfloor n/2\rfloor$, $B_{n,k,\ell}$ has rank $\binom{n}{\ell}$. In \cite[Theorem 4.3]{dalfo2021laplacian}, it was shown that for any $1\leq k\leq n$,
\begin{equation}\label{eq:BLLB}
        B_{n,k,k-1} L_{k-1}(G)= L_k(G) B_{n,k,k-1}.
\end{equation}
Equation \eqref{eq:BLLB} implies that both $\text{Im}(B_{n,k,k-1})$ and $\text{Im} (B_{n,k,k-1})^{\perp}= \text{Ker} (B_{n,k,k-1}^T)$ are invariant subspaces of $L_k(G)$. Moreover, for $1\leq k\leq \lfloor n/2\rfloor$, using \eqref{eq:BLLB} and the fact that $B_{n,k,k-1}$ has full column rank, we obtain that if $\phi_1,\ldots,\phi_{\binom{n}{k-1}}$ form a basis of $\Rea^{\binom{n}{k-1}}$ consisting of eigenvectors of $L_{k-1}(G)$, with eigenvalues $\lambda_1,\ldots,\lambda_{\binom{n}{k-1}}$ respectively, then $B_{n,k,k-1}\phi_1,\ldots, B_{n,k,k-1}\phi_{\binom{n}{k-1}}$ form a basis of $\text{Im}(B_{n,k,k-1})$ consisting of eigenvectors of $L_k(G)$, with the same eigenvalues $\lambda_1,\ldots,\lambda_{\binom{n}{k-1}}$ (see \cite[Corollary 4.5]{dalfo2021laplacian}). In particular, the spectrum of $L_{k-1}(G)$ is contained (including multiplicities) in the spectrum of $L_k(G)$.
As immediate consequences, we obtain the following useful results:

\begin{lemma}
\label{lemma:non_trivial} Let $1\leq k\leq \lfloor n/2 \rfloor$, and let $\lambda$ be an eigenvalue of $L_k(G)$. Then, $\lambda$ is non-trivial if and only if there exists an eigenvector $\phi$ of $L_k(G)$ with eigenvalue $\lambda$ satisfying $B^{T}_{n,k,k-1}\phi=0$.
\end{lemma}

\begin{lemma}\label{lemma:lambda_min}
Let $1\leq k\leq \lfloor n/2 \rfloor$. Then
\[
    \eigen{min}{k}{G}= \min\left\{\frac{\angles{L_k(G)\phi}{\phi}}{\|\phi\|^2} \, : \, 0\neq \phi\in\Rea^{\binom{n}{k}},\, B^T_{n,k,k-1}\phi=0\right\}.
\]
\end{lemma}

\section{A Garland-type argument}\label{sec:garland}

In this section we prove our main results, Theorems \ref{thm:step} and \ref{thm:main_intro}. 
Let $G=([n],E)$ be a graph, and let $2\leq k\leq \lfloor n/2 \rfloor$. 
Let $L=L(G)$ and $L_k=L_k(G)=L(F_k(G))$. We will denote  the edge set of $F_k(G)$ by $E_k$. Moreover, for $\sigma\in\binom{[n]}{k}$, let $d_k(\sigma)=d_{F_k(G)}(\sigma)$ be the degree of $\sigma$ in $F_k(G)$.

Let $\phi\in \Rea^{\binom{n}{k}}$. For any $u\in [n]$, we define $\phi_{u}\in \Rea^{\binom{n}{k-1}}$ by
\[
\phi_{u}(\tau)=\begin{cases}
    \phi(\tau\cup\{u\}) & \text{ if } u\notin \tau,\\
    0 & \text{ if } u\in\tau,
    \end{cases}
\]
for any $\tau\in\binom{[n]}{k-1}$.
For a set $\sigma\in\binom{[n]}{k}$, denote 
$
    E_{\sigma}=\{e\in E:\, e\subset \sigma\}.
$
We define a diagonal matrix $D_k\in \Rea^{\binom{n}{k}\times \binom{n}{k}}$ by
\begin{equation}\label{eq:dk}
    (D_k)_{\sigma,\tau}=\begin{cases}
        |E_{\sigma}| & \text{ if } \sigma=\tau,\\
        0 & \text{ otherwise,}
    \end{cases}
\end{equation}
for all $\sigma,\tau\in\binom{[n]}{k}$.
Theorem \ref{thm:step} will follow from the following identity:
\begin{proposition}\label{prop:form2}
    Let $\phi\in \Rea^{\binom{n}{k}}$. Then
    \[
    (k-1)\angles{L_k\phi}{\phi}= \left(\sum_{u=1}^n \angles{L_{k-1}\phi_{u}}{\phi_{u}}\right) - 2\angles{D_k\phi}{\phi}.
    \]
\end{proposition}

For the proof of Proposition \ref{prop:form2}, we will need the following result about sums of degrees in $F_k(G)$.

\begin{lemma}\label{lemma:degree}
Let $\sigma\in\binom{[n]}{k}$. Then,
\[
\sum_{u\in\sigma} d_{k-1}(\sigma\setminus\{u\})= (k-1)d_k(\sigma)+ 2|E_{\sigma}|.
\]
\end{lemma}
\begin{proof}
Note that, for any $0\leq j\leq n$ and any $\eta\in\binom{[n]}{j}$,  $d_j(\eta)=|\{e\in E:\, |e\cap\eta|=1\}|$. Therefore,
\begin{align*}
\sum_{u\in\sigma} d_{k-1}(\sigma\setminus\{u\})&=\sum_{u\in\sigma}\sum_{v\in\sigma\setminus\{u\}} \sum_{\substack{w\in([n]\setminus\sigma)\cup\{u\},\\ \{v,w\}\in E}} 1  =\sum_{u\in\sigma}\sum_{v\in\sigma\setminus\{u\}} \sum_{\substack{w\in[n]\setminus\sigma,\\ \{v,w\}\in E}} 1
+\sum_{u\in\sigma}\sum_{\substack{v\in\sigma\setminus\{u\},\\ \{u,v\}\in E}} 1
\\
&= (k-1)\sum_{v\in\sigma} \sum_{\substack{w\in[n]\setminus\sigma,\\ \{v,w\}\in E}} 1 +\sum_{u\in\sigma}\sum_{\substack{v\in\sigma\setminus\{u\},\\ \{u,v\}\in E}} 1 = (k-1)d_k(\sigma) + 2 |E_{\sigma}|.
\end{align*}
\end{proof}

\begin{proof}[Proof of Proposition \ref{prop:form2}]
We have
 \begin{align}\label{eq:form1}
\angles{L_k\phi}{\phi} &=\sum_{\{\sigma,\tau\} \in E_k}(\phi(\sigma)-\phi(\tau))^2
=\sum_{\sigma\in\binom{[n]}{k}} d_{k}(\sigma) \phi(\sigma)^2 -2\sum_{\{\sigma,\tau\} \in E_k}\phi(\sigma)\phi(\tau) \nonumber
\\
&=\sum_{\sigma\in\binom{[n]}{k}}d_{k}(\sigma) \phi(\sigma)^2 -2\sum_{\eta\in\binom{[n]}{k-1}}\sum_{\substack{\{v,w\}\in E,\\ v,w\notin\eta
    }} \phi(\eta\cup\{v\})\phi(\eta\cup\{w\}).
 \end{align}   
 Similarly,
\begin{align}\label{eq:form2}
    \sum_{u=1}^n &\angles{L_{k-1}\phi_{u}}{\phi_{u}}
    = \sum_{u=1}^n \sum_{\tau\in\binom{[n]}{k-1}}  d_{k-1}(\tau) \phi_{u}(\tau)^2-2 \sum_{u=1}^n \sum_{\eta\in\binom{[n]}{k-2}} \sum_{\substack{\{v,w\}\in E,\\ v,w\notin\eta}}\phi_u(\eta\cup\{v\})\phi_u(\eta\cup\{w\}) \nonumber
    \\
    &= \sum_{\tau\in\binom{[n]}{k-1}}\sum_{u\in[n]\setminus\tau}  d_{k-1}(\tau) \phi(\tau\cup\{u\})^2-2 \sum_{\eta\in\binom{[n]}{k-2}}\sum_{u\in[n]\setminus \eta}\sum_{\substack{\{v,w\}\in E,\\ v,w\notin\eta\cup\{u\}}} \phi(\eta\cup\{u,v\})\phi(\eta\cup\{u,w\}) \nonumber
\\
    &= \sum_{\sigma\in\binom{[n]}{k}} \left(\sum_{u\in\sigma} d_{k-1}(\sigma\setminus\{u\})\right) \phi(\sigma)^2-2 (k-1)\sum_{\tau\in\binom{[n]}{k-1}} \sum_{\substack{\{v,w\}\in E,\\v,w\notin\tau}} \phi(\tau\cup\{v\})\phi(\tau\cup\{w\}).
\end{align}
By Lemma \ref{lemma:degree}    we have, for all $\sigma\in\binom{[n]}{k}$,  
\[
\sum_{u\in\sigma}d_{k-1}(\sigma\setminus\{u\})= (k-1)d_k(\sigma)+ 2|E_{\sigma}|.
\]
Therefore, by \eqref{eq:form1} and \eqref{eq:form2}, we obtain
\[
(k-1)\angles{L_k \phi}{\phi} = \sum_{u=1}^n \angles{L_{k-1}\phi_{u}}{\phi_{u}} - 2 \sum_{\sigma\in\binom{[n]}{k}} |E_{\sigma}|\phi(\sigma)^2
=  \sum_{u=1}^n \angles{L_{k-1}\phi_{u}}{\phi_{u}}- 2\angles{D_k\phi}{\phi}.
\]
\end{proof}

We will also need the following lemma.

\begin{lemma}
\label{lemma:norm}
Let $\phi\in \Rea^{\binom{n}{k}}$. Then
\[
    \sum_{u=1}^n \|\phi_{u}\|^2 = k\|\phi\|^2.
\]
\end{lemma}
\begin{proof}
\[
     \sum_{u=1}^n \|\phi_{u}\|^2 = \sum_{u=1}^n \sum_{\tau\in\binom{[n]}{k-1}} 
     \phi_{u}(\tau)^2 =
     \sum_{\tau\in\binom{[n]}{k-1}}  \sum_{u\in[n]\setminus\tau}
     \phi(\tau\cup\{u\})^2
     =\sum_{\sigma\in\binom{[n]}{k}}  \sum_{u\in\sigma}
     \phi(\sigma)^2
     =k\sum_{\sigma\in\binom{[n]}{k}} 
     \phi(\sigma)^2=k\|\phi\|^2.
\]
\end{proof}

We can now prove Theorem \ref{thm:step}:

\begin{proof}[Proof of Theorem \ref{thm:step}]

Let $\lambda=\eigen{max}{k}{G}$, and let $\phi\in\Rea^{\binom{n}{k}}$ be an eigenvector of $L_k$ with eigenvalue $\lambda$.
By Proposition \ref{prop:form2}, we have
\[
    (k-1)\lambda \|\phi\|^2 = (k-1)\angles{L_k \phi}{\phi} =  \sum_{u=1}^n \angles{L_{k-1}\phi_{u}}{\phi_{u}}- 2\angles{D_k\phi}{\phi}.
\]

Since $|E_{\sigma}|\geq 0$ for all $\sigma\in\binom{[n]}{k}$, we have $\angles{D_k\phi}{\phi}\geq 0$, and therefore 
\[
(k-1)\lambda \|\phi\|^2\leq 
 \sum_{u=1}^n \angles{L_{k-1}\phi_{u}}{\phi_{u}}
\leq \sum_{u=1}^n \eigen{max}{k-1}{G} \|\phi_{u}\|^2 = k \eigen{max}{k-1}{G} \|\phi\|^2,
\]
where the last equality follows from Lemma \ref{lemma:norm}.
Hence, we obtain $\lambda\leq k \eigen{max}{k-1}{G}/(k-1)$, as wanted.

Now, let $\lambda=\eigen{min}{k}{G}$. By Lemma \ref{lemma:non_trivial}, since $\lambda$ is non-trivial, there is an eigenvector $\phi$ of $L_k$ with eigenvalue $\lambda$ such that $B^T_{n,k,k-1}\phi=0$. 
We will show that, for any $u\in[n]$, $B^{T}_{n,k-1,k-2}\phi_u=0$. Let $u\in[n]$ and $\eta\in \binom{[n]}{k-2}$. If $u\in\eta$, we have
\[
    B^{T}_{n,k-1,k-2}\phi_u(\eta)= \sum_{\substack{\tau\in\binom{[n]}{k-1},\\ \eta\subset\tau}} \phi_u(\tau) = 0,
\]
by the definition of $\phi_u$. If $u\notin\eta$, then
\[
    B^{T}_{n,k-1,k-2}\phi_u(\eta)= \sum_{\substack{\tau\in\binom{[n]}{k-1},\\ \eta\subset\tau}} \phi_u(\tau) = 
  \sum_{\substack{\tau\in\binom{[n]}{k-1},\\ \eta\subset\tau,\, u\notin\tau}} \phi(\tau\cup\{u\}) =\sum_{\substack{\sigma\in\binom{[n]}{k},\\ \eta\cup\{u\}\subset\sigma}} \phi(\sigma)= B^T_{n,k,k-1}\phi(\eta\cup\{u\})=0.
\]
Therefore, by Lemma \ref{lemma:lambda_min}, $\angles{L_{k-1}\phi_{u}}{\phi_{u}}\geq \eigen{min}{k-1}{G}\|\phi_u\|^2$ for all $u\in[n]$. Moreover, since $|E_{\sigma}|\leq \binom{k}{2}$ for all $\sigma\in\binom{[n]}{k}$, we have $\angles{D_k\phi}{\phi}\leq \binom{k}{2}\|\phi\|^2$, and thus, by Proposition \ref{prop:form2}, we obtain
\begin{align*}
(k-1)\lambda\|\phi\|^2 &=(k-1)\angles{L_k\phi}{\phi}= 
\sum_{u=1}^n \angles{L_{k-1}\phi_{u}}{\phi_{u}}- 2\angles{D_k\phi}{\phi}
\\
&\geq \sum_{u=1}^n\eigen{min}{k-1}{G}\|\phi_{u}\|^2 - k(k-1)\|\phi\|^2
 = \left( k\eigen{min}{k-1}{G}-k(k-1)\right) \|\phi\|^2,
\end{align*}
where the last equality follows from Lemma \ref{lemma:norm}. Hence, $\lambda \geq k\eigen{min}{k-1}{G}/(k-1)-k
$.    
\end{proof}

Finally, we obtain Theorem \ref{thm:main_intro} by an inductive application of Theorem \ref{thm:step}.

\begin{proof}[Proof of Theorem \ref{thm:main_intro}]
We will show that \[\eigen{max}{k}{G}\leq k \lambda_n(L)\] and \[\eigen{min}{k}{G} \geq k(\lambda_2(L)-k+1).\]
We argue by induction on $k$.
For $k=1$ the claim is trivial. Let $k\geq 2$, and assume that
\[
  \eigen{max}{k-1}{G} \leq (k-1)\lambda_n(L)
\]
and
\[
\eigen{min}{k-1}{G}\geq (k-1)(\lambda_2(L)-(k-1)+1).
\]
Then, by Theorem \ref{thm:step}, we obtain
\[
\eigen{max}{k}{G} \leq \frac{k}{k-1}\eigen{max}{k-1}{G} \leq k\lambda_n(L),
\]
and
\[
\eigen{min}{k}{G} \geq \frac{k}{k-1}\eigen{min}{k-1}{G}-k \geq k(\lambda_2(L)-k+1). 
\]
\end{proof}

\section{Extremal examples}\label{sec:examples}

The next result shows that the upper and lower bounds in Theorem \ref{thm:main_intro} are sharp.

\begin{proposition}\label{prop:examples}
Let $m\geq k$, and let $n$ be divisible by $m$. 
Let $G$ be the union of $m$ disjoint cliques, each of size $n/m$. Then, the maximal eigenvalue of $L_k(G)$ is exactly $k n/m=k\lambda_{n}(L(G))$.

Let $\bar{G}$ be the complement graph of $G$, namely the complete balanced $m$-partite graph with sides of size $n/m$. Then, the minimal non-trivial eigenvalue of $L_k(\bar{G})$ is exactly $k((m-1)n/m-k+1)=k(\lambda_2(L(\bar{G}))-k+1)$.
\end{proposition}

Proposition \ref{prop:examples} follows from an argument similar to the one in  \cite[Theorem 7.2(iv)]{dalfo2021laplacian}. For completeness, we include a proof.
We will need the following result due to Fiedler. Recall that, given graphs $G=(V,E)$ and $G'=(V',E')$, the Cartesian product $G\square G'$ is the graph on vertex set $V\times V'$ with edges of the form $\{(u,u'),(v,v')\}$, where either $u=v$ and $u'$ is adjacent to $v'$ in $G'$, or $u'=v'$ and $u$ is adjacent to $v$ in $G$.

\begin{lemma}[Fiedler {\cite[3.4]{fiedler1973algebraic}}]\label{lemma:cartesian}
For $1\leq i\leq m$, let $H_i$ be a graph on $n_i$ vertices, and let $\lambda_1^i\leq \cdots \leq \lambda_{n_i}^i$ be its Laplacian eigenvalues. Then, the Laplacian eigenvalues of the Cartesian product $H_1\square \cdots \square H_m$ are
\[
  \left\{  \sum_{i=1}^m \lambda_{t_i}^{i} :\, 1\leq t_i\leq n_i \,\,\forall i\in[m] \right\}.
\]
\end{lemma}

We will also need the following result, describing the Laplacian spectra of token graphs of a complete graph (also known as Johnson graphs).

\begin{lemma}[See {\cite[Thm. 6.3.2, Thm. 6.3.3]{godsil2016erdos}}, {\cite[Eq. 19]{dalfo2021laplacian}}]
\label{lemma:johnson}
    Let $K_n$ be the complete graph on $n$ vertices. Then, the eigenvalues of $L_k(K_n)$ are
\[
    \{j(n-j+1) :\, 0\leq j\leq k\}.
\]
Moreover, for every $0\leq j\leq k$, the eigenspace corresponding to the eigenvalue $j(n-j+1)$ is the orthogonal complement of $\text{Im}( B_{n,k,j-1})$ in $\text{Im}(B_{n,k,j})$, and therefore 
the eigenvalue $j(n-j+1)$ has multiplicity $\binom{n}{j}-\binom{n}{j-1}$. 
In particular, the eigenvectors corresponding to the eigenvalue $k(n-k+1)$ are exactly the vectors $\phi\in\Rea^{\binom{n}{k}}$ satisfying $B^{T}_{n,k,k-1}\phi=0$.
\end{lemma}

Finally, we will need the following result relating the spectrum of $L_k(G)$ and that of $L_k(\bar{G})$, which follows from the proof of \cite[Theorem 6.2]{dalfo2021laplacian}:

\begin{lemma}\label{lemma:complement}
Let $G$ be a graph on $n$ vertices, and let $\lambda$ be a non-trivial eigenvalue of $L_k(G)$. Then $\lambda=k(n-k+1)-\mu$, where $\mu$ is some non-trivial eigenvalue of $L_k(\bar{G})$.
\end{lemma}
\begin{proof}
    Let $\phi$ be an eigenvector of $L_k(G)$ with eigenvalue $\lambda$. By Lemma \ref{lemma:non_trivial}, we can assume $B_{n,k,k-1}^T\phi=0$. 
    It was shown in \cite[Theorem 6.2]{dalfo2021laplacian} that $\phi$ is also an eigenvector of $L_k(\bar{G})$ with eigenvalue $\mu$, and an eigenvector of $L_k(K_n)$ with eigenvalue $\xi$, for some $\mu$ and $\xi$ such that $\lambda+\mu=\xi$. Since $B_{n,k,k-1}^T \phi=0$, by Lemma \ref{lemma:johnson} we have $\xi=k(n-k+1)$. We obtain $\lambda=k(n-k+1)-\mu$. Furthermore, since $\phi$ is an eigenvector of $L_k(\bar{G})$ with eigenvalue $\mu$ satisfying $B_{n,k,k-1}^T \phi=0$, by Lemma \ref{lemma:non_trivial} $\mu$ is a non-trivial eigenvalue of $L_k(\bar{G})$.
\end{proof}

\begin{proof}[Proof of Proposition \ref{prop:examples}]

Let $G_1,\ldots,G_m$ be the connected components of $G$, each isomorphic to the complete graph on $n/m$ vertices.
Let \[\mathcal{I}=\left\{(k_1,\ldots,k_m):\, 0\leq k_i\leq n/m,\, \sum_{i=1}^m k_i=k\right\}.\] Then, for each $(k_1,\ldots,k_m)\in\mathcal{I}$, $F_k(G)$ has a connected component isomorphic to $F_{k_1}(G_1)\square \cdots\square F_{k_m}(G_m)$ (see proof of Corollary 6.4 in \cite{dalfo2021laplacian}). By Lemma \ref{lemma:cartesian} and Lemma \ref{lemma:johnson}, every eigenvalue of $L_k(G)$ is of the form
\[
  \sum_{i=1}^m j_i(n/m-j_i+1)=\left(\sum_{i=1}^m j_i\right) n/m-\sum_{i=1}^m j_i(j_i-1),
\]
for $(k_1,\ldots,k_m)\in\mathcal{I}$ and $0\leq j_i\leq k_i$ for all $1\leq i\leq m$.

Since $m\geq k$, we can choose
each $k_i$ to be either $0$ or $1$, and $j_i=k_i$ for all $i$, to obtain an eigenvalue 
\begin{equation}\label{eq:lambdamaxexample}
\lambda=\left(\sum_{i=1}^m k_i\right) n/m-\sum_{i=1}^m k_i(k_i-1)= kn/m= k \lambda_n(L(G)).
\end{equation}
By Theorem \ref{thm:main_intro}, this is the maximal eigenvalue of $L_k(G)$. Furthermore, note that, by Theorem \ref{thm:main_intro}, this is a non-trivial eigenvalue of $L_k(G)$ (otherwise, it is an eigenvalue of $L_{k-1}(G)$ larger than $(k-1)\lambda_n(L(G))$, a contradiction).

Now, let $\bar{G}$ be the complement of $G$. Then $\lambda_2(L(\bar{G}))=n-\lambda_n(L(G))= (m-1)n/m$. By Lemma \ref{lemma:complement}, $\lambda=k(n-k+1)-\mu$, where $\mu$ is some non-trivial eigenvalue of $L_k(G)$. By \eqref{eq:lambdamaxexample}, we have
\begin{align*}
    \mu&= k(n-k+1)-\lambda= k(n-k+1)-kn/m
    \\
    &=k((m-1)n/m-k+1)= k(\lambda_2(L(\bar{G}))-k+1).
\end{align*}
By Theorem \ref{thm:main_intro}, this is the minimal non-trivial eigenvalue of $L_k(\bar{G})$.
\end{proof}

\bibliographystyle{abbrv}
\bibliography{biblio}

\begin{thebibliography}{10}

\bibitem{aharoni2005eigenvalues}
R.~Aharoni, E.~Berger, and R.~Meshulam.
\newblock Eigenvalues and homology of flag complexes and vector representations
  of graphs.
\newblock {\em Geometric and Functional Analysis}, 15(3):555--566, 2005.

\bibitem{alavi1991double}
Y.~Alavi, M.~Behzad, P.~Erd\H{o}s, and D.~R. Lick.
\newblock Double vertex graphs.
\newblock {\em J. Combin. Inform. System Sci}, 16(1):37--50, 1991.

\bibitem{alavi2002survey}
Y.~Alavi, D.~R. Lick, and J.~Liu.
\newblock Survey of double vertex graphs.
\newblock {\em Graphs and Combinatorics}, 18:709--715, 2002.

\bibitem{audenaert2007symmetric}
K.~Audenaert, C.~Godsil, G.~Royle, and T.~Rudolph.
\newblock Symmetric squares of graphs.
\newblock {\em Journal of Combinatorial Theory, Series B}, 97(1):74--90, 2007.

\bibitem{ballmann1997l2}
W.~Ballmann and J.~{\'S}wiatkowski.
\newblock On {$L^2$}-cohomology and property ({T}) for automorphism groups of
  polyhedral cell complexes.
\newblock {\em Geometric and Functional Analysis}, 7(4):615--645, 1997.

\bibitem{caputo2010proof}
P.~Caputo, T.~Liggett, and T.~Richthammer.
\newblock Proof of {A}ldous’ spectral gap conjecture.
\newblock {\em Journal of the American Mathematical Society}, 23(3):831--851,
  2010.

\bibitem{dalfo2021laplacian}
C.~Dalf{\'o}, F.~Duque, R.~Fabila-Monroy, M.~A. Fiol, C.~Huemer, A.~L.
  Trujillo-Negrete, and F.~Z. Mart{\'\i}nez.
\newblock On the {L}aplacian spectra of token graphs.
\newblock {\em Linear Algebra and its Applications}, 625:322--348, 2021.

\bibitem{dalfo2022algebraic}
C.~Dalf{\'o} and M.~A. Fiol.
\newblock On the algebraic connectivity of token graphs.
\newblock {\em arXiv preprint arXiv:2209.01030}, 2022.

\bibitem{fabila2012token}
R.~Fabila-Monroy, D.~Flores-Pe{\~n}aloza, C.~Huemer, F.~Hurtado, J.~Urrutia,
  and D.~R. Wood.
\newblock Token graphs.
\newblock {\em Graphs and Combinatorics}, 28:365--380, 2012.

\bibitem{fiedler1973algebraic}
M.~Fiedler.
\newblock Algebraic connectivity of graphs.
\newblock {\em Czechoslovak Mathematical Journal}, 23(2):298--305, 1973.

\bibitem{garland1973p}
H.~Garland.
\newblock $p$-adic curvature and the cohomology of discrete subgroups of
  $p$-adic groups.
\newblock {\em Annals of Mathematics}, pages 375--423, 1973.

\bibitem{godsil2016erdos}
C.~Godsil and K.~Meagher.
\newblock {\em Erd\H{o}s-Ko-Rado theorems: algebraic approaches}.
\newblock Cambridge University Press, 2016.

\bibitem{gottlieb1966certain}
D.~H. Gottlieb.
\newblock A certain class of incidence matrices.
\newblock {\em Proceedings of the American Mathematical Society},
  17(6):1233--1237, 1966.

\bibitem{graver1973module}
J.~E. Graver and W.~Jurkat.
\newblock The module structure of integral designs.
\newblock {\em Journal of Combinatorial Theory, Series A}, 15(1):75--90, 1973.

\bibitem{johns1988generalized}
G.~L. Johns.
\newblock {\em Generalized distance in graphs}.
\newblock PhD thesis, Western Michigan University, 1988.

\bibitem{lew2020spectral}
A.~Lew.
\newblock The spectral gaps of generalized flag complexes and a geometric
  {H}all-type theorem.
\newblock {\em International Mathematics Research Notices},
  2020(11):3364--3395, 2020.

\bibitem{ouyang2019computing}
Y.~Ouyang.
\newblock Computing spectral bounds of the {H}eisenberg ferromagnet from
  geometric considerations.
\newblock {\em Journal of Mathematical Physics}, 60(7):071901, 2019.

\bibitem{reyesalgebraic}
M.~Reyes, C.~Dalf{\'o}, M.~A. Fiol, and A.~Messegu{\'e}.
\newblock On the algebraic connectivity of token graphs of a cycle.
\newblock {\em Proceedings of the 39th {E}uropean {W}orkshop on {C}omputational
  {G}eometry}, 2023.

\bibitem{SPITZER1970246}
F.~Spitzer.
\newblock Interaction of {M}arkov processes.
\newblock {\em Advances in Mathematics}, 5(2):246--290, 1970.

\bibitem{wilson1973necessary}
R.~M. Wilson.
\newblock The necessary conditions for $t$-designs are sufficient for
  something.
\newblock {\em Utilitas Math}, 4:207--215, 1973.

\bibitem{wright1992n}
V.~Wright.
\newblock $n$-tuple vertex graphs.
\newblock Master's thesis, Emory University, Atlanta, 1992.

\bibitem{zuk2003property}
A.~{\.Z}uk.
\newblock Property ({T}) and {K}azhdan constants for discrete groups.
\newblock {\em Geometric \& Functional Analysis GAFA}, 13(3):643--670, 2003.

\end{thebibliography}

\end{document}